\documentclass[11pt,reqno]{amsart}
\usepackage{amsmath,amssymb,amsthm,amsfonts,enumerate,hyperref}

\newtheorem{theorem}{Theorem}
\newtheorem{cor}[theorem]{Corollary}

\theoremstyle{definition}
\newtheorem{definition}[theorem]{Definition}
\newtheorem{rem}[theorem]{Remark}

\newcommand{\abs}[1]{\lvert #1\rvert}
\newcommand{\eps}{\varepsilon}
\newcommand{\Lap}{\Delta}
\newcommand{\D}{\nabla}

\newcommand{\Ric}{\mathrm{Ric}}

\newcommand{\RR}{\mathbb{R}}

\begin{document}

\title{A mass-decreasing flow in dimension three}

\begin{abstract}
In this article, we introduce a mass-decreasing flow for asymptotically flat three-manifolds with nonnegative scalar curvature. This flow is defined by iterating a suitable Ricci flow with surgery and conformal rescalings and has a number of nice properties. In particular, wormholes pinch off and nontrivial spherical space forms bubble off in finite time. Moreover, a noncompact variant of the Perelman-energy is monotone along the flow. Assuming a certain inequality between the mass and this Perelman-energy a priori, we can prove that the flow squeezes out all the initial mass.
\end{abstract}

\author{Robert Haslhofer}

\address{Department of Mathematics, ETH Z\"{u}rich, Switzerland}

\email{robert.haslhofer@math.ethz.ch}
\maketitle

\section{Introduction}
Let $(M,g_{ij})$ be an asymptotically flat three-manifold with nonnegative integrable scalar curvature. The ADM-mass \cite{ADM,Ba} from general relativity is defined as
\begin{equation}\label{admmass}
m(g):=\lim_{r\to\infty}\int_{S_r}\left(\partial_jg_{ij}-\partial_ig_{jj}\right)dA^{i}.
\end{equation}
By the positve mass theorem, the mass is always nonnegative and vanishes only for flat space. Beautiful proofs employing a variety of techniques have been discovered \cite{SY,W,HI}. The first one, due to Schoen and Yau, is based on a very nice argument by contradiction using the stability inequality for minimal surfaces. The second one, discovered by Witten, utilizes beautiful identities for Dirac spinors. A third remarkable proof of the positive mass theorem (and in fact of the Penrose inequality), due to Huisken and Ilmanen, is based on the inverse mean curvature flow.\\

It is of great analytic, geometric, and physical interest to investigate how Hamilton's Ricci flow \cite{Ham} interacts with the positive mass theorem. This relationship has been studied in \cite{GHMS}, \cite{DM}, \cite{OW}, and \cite{H}. In particular, asymptotic flatness and nonnegative integrable scalar curvature are preserved, and the Ricci flow can be used to prove the rigidity statement in the positive mass theorem. However, the mass (although not the quasilocal mass) is constant along the Ricci flow and this leads to Bray's intriguing question whether there exists a geometric flow that decreases the mass. Some hope that such a flow might exist comes from the deep relationship between the mass and geometric flows in the proofs of the Penrose inequality by Huisken-Ilmanen, Bray, and Bray-Lee \cite{HI,Br,BL}.\\

The purpose of this article is to confirm this hope, and to introduce and investigate a mass-decreasing flow in dimension three. We have announced the discovery of this flow in \cite{H}, where we also sketched some concepts and ideas showing an intriguing relationship between the mass, the Perelman-energy, and the stability of Ricci-flat spaces. Motivated by that, our mass-decreasing flow is defined by iterating a suitable Ricci flow with surgery and conformal rescalings. The point is, that conformal rescalings to scalar flat metrics squeeze out of the manifold as much mass as possible. However, unless the manifold is flat, the scalar curvature becomes strictly positive again under the Ricci flow and thus the mass can be decreased even more by another conformal rescaling. This process can be iterated forever.\\

The idea that conformal transformations can be used to decrease the mass and that deformations in direction of the Ricci curvature can be used to increase the scalar curvature again goes back to the fundamental work of Schoen and Yau \cite{SY}. What is more recent, is the precise geometric-analytic understanding of the Ricci flow in dimension three due to the revolutionary work of Perelman \cite{P1,P2,P3}, see \cite{KL,MT,CZ,BBBMP} for detailed expositions. In particular, we use Perelman's existence theorem for the Ricci flow with surgery, or more precisely a very nice variant for noncompact manifolds due to Bessi\`eres-Besson-Maillot \cite{BBM}.\\

In Section \ref{DefandLTE}, we give the precise definition of our flow. We prove that it exists for all times, and preserves asymptotic flatness and nonnegative integrable scalar curvature (Theorem \ref{thmLTE}). Most importantly, the mass $m(g(t))$ is strictly decreasing along the flow. Since it is also bounded below by zero, it has a nonnegative limit for $t\to\infty$. We conjecture that the flow always squeezes out all the initial mass, i.e. $\lim_{t\to\infty}m(g(t))=0$. Support for this conjecture comes from the analysis of the long-time behavior of the mass-decreasing flow, which we carry out in the following two sections.\\

In Section \ref{longtime1}, we treat the toplogical aspects of the long-time behavior. We prove that nontrivial topology becomes extinct in finite time, analogous to the extinction theorem for closed manifolds admitting a metric of positive scalar curvature due to Perelman \cite{P3} and Colding-Minicozzi \cite{CM1,CM2}. The manifolds in consideration are diffeomorphic to $\RR^3$ with finitely many $S^1\times S^2$ and $S^3/\Gamma$ pieces attached (see e.g. Corollary \ref{cortop}). Thus, the extinction result can be rephrased in more physical words saying that wormholes pinch off and nontrivial spherical space forms bubble off in finite time. In fact, this happens in time at most $T={A_0}/{4\pi}$, where $A_0$ is the area of the largest outermost minimal two-sphere in the initial manifold (Theorem \ref{thmfinext}).\\

In Section \ref{longtime2}, we make partial progress towards understanding the geometric-analytic aspects of the long-time behavior. Following the general principle that monotonicity formulas are a crucial tool, we encounter the problem that Perelman's $\lambda$-energy is in fact identically zero in the case of asymptotically flat manifolds with nonnegative scalar curvature. We overcome this difficulty by using a suitable noncompact variant of the Perelman-energy, an energy-functional $\lambda_{AF}$ that we recently introduced in \cite{H}. This energy $\lambda_{AF}$ is nontrivial in the asymptotically flat setting, and we prove that it satisfies a Perelman-like monotonicity formula (Theorem \ref{thmmon}, Remark \ref{surgthm}). Assuming a certain inequality between the mass and $\lambda_{AF}$ a priori, we can prove that the mass-decreasing flow indeed squeezes out all the initial mass (Theorem \ref{thmapriori}).\\

In Section \ref{contlimit}, we derive the limiting equations that formally arise when our iteration parameter $\eps$ is sent to zero.\footnote{Shortly after the author posted the first version of this paper on arXiv, Peng Lu, Jie Qing and Yu Zheng posted a very interesting note where they proved short time existence for the nonlocal limiting equation \cite{LQZ}. We have now added Section \ref{contlimit} in this new version of our paper to clarify the relationship with their paper and also to indicate the possible relevance for the analysis of the long-time behavior.} However, we actually prefer to \emph{avoid} to really take the limit $\eps\to 0$, since we want to use our long-time existence result that relies on the theory of Ricci flow with surgery. Nevertheless, taking $\eps$ small enough one is still close to the limiting equations and therefore the limiting equations might be useful to analyze the long time behavior in the general case without a priori assumptions.\\

Finally, in Section \ref{openprob}, we collect some open problems and questions.\\

\emph{Acknowledgements.} After giving a talk at the conference `Geometric flows in mathematics and physics' at BIRS Banff, I learned that Lars Andersson and Hugh Bray had been thinking about related issues. I greatly thank both of them for sharing with me their proposals, questions and ideas. I also thank Jacob Bernstein, Simon Brendle, Michael Eichmair, Gerhard Huisken, Tom Ilmanen, Richard Schoen and Burkhard Wilking for very interesting discussions and suggestions, and the Swiss National Science Foundation for partial financial support.

\section{Definition of the flow and long-time existence}\label{DefandLTE}

We will use the following existence theorem for surgical Ricci flow solutions due to Bessi\`eres-Besson-Maillot \cite{BBM} which relies heavily on the work of Perelman \cite{P1,P2}.

\begin{theorem}[{Bessi\`eres-Besson-Maillot \cite[Theorem 5.5]{BBM}}]\label{BBMthm}
For every $\iota>0$, $K<\infty$ and $T<\infty$, there exist $r,\delta,\kappa> 0$ such that for any oriented complete Riemannian 3-manifold $(M_0,g_0)$ with $\abs{Rm}\leq K$ and injectivity radius at least $\iota$, there exists an $(r,\delta,\kappa)$-surgical Ricci flow solution defined on $[0,T]$ with initial conditions $(M_0,g_0)$.
\end{theorem}

For full details about surgical Ricci flow, please see \cite{BBM}, but let us give a quick overview here and collect some facts that we will use later.\\

A surgical Ricci flow solution is a sequence of Ricci flows $(M_i,g_i(t))_{t\in[t_i,t_{i+1}]}$ with $0=t_0<t_1<\ldots$, such that $M_{i+1}$ is obtained from $M_i$ by splitting along embedded two-spheres, gluing in standard caps, and throwing away connected components covered entirely by canonical neighborhoods.\\

The main difference with Perelman's original surgery procedure is that this one is done \emph{before} the singular time, namely when the supremum of the scalar curvature reaches a certain threshold $\Theta$.\\

Since the curvature is pinched towards positive and $R\leq \Theta$ by construction, we have uniform curvature bounds along the flow. Moreover, the surgeries are done in such a way that the supremum of the scalar curvature drops by at least a factor $1/2$. This ensures that the surgery times don't accumulate. Also, the infimum of the scalar curvature is nondecreasing.\\

The most important parameters in the construction are the canonical neighborhood scale $r$, the noncollapsing parameter $\kappa$, the surgery parameter $\delta$, and the threshold $\Theta$. These parameters have the following significance: First, if $R(x,t)\geq r^{-2}$ at some point $(x,t)$ in a surgical solution, then there exists a canonical neighborhood of $(x,t)$. Second, the solution is $\kappa$-noncollapsed at scales less than one. Third, the surgeries are performed inside very small $\delta$-necks. Fourth and finally, $R\leq \Theta$ along the flow and the surgeries are done when the scalar curvature reaches the treshold $\Theta$.\\

Having completed this very quick overview, we can now define a new geometric flow as follows:

\begin{definition}[mass-decreasing flow] Let $(M,g_0)$ be an (oriented, smooth, complete, connected) asymptotically flat three manifold of order one, with nonnegative integrable scalar curvature, and fix a parameter $\eps>0$.
\begin{itemize}
\item Let $(M(t),g(t))_{t\in[0,\eps]}$ be the surgical Ricci flow solution of Theorem \ref{BBMthm} starting at $g_0$, with all connected components except the one containing the asymptotically flat end thrown away.
\item As a second step, solve the elliptic equation
\begin{equation}
\left(-8\Lap_{g(\eps)}+R_{g(\eps)}\right)w_1=0,\qquad\quad w_1\to 1\quad \textrm{at}\quad \infty,
\end{equation}
and set $g_1:=w_1^4g(\eps)$.
\item Finally, let $(M(\eps),g_1)$ be the new initial condition and iterate the above procedure. The concatenation `flow, conformal rescaling, flow, conformal rescaling, ...' gives an evolution $(M(t),g(t))_{t\in[0,\infty)}$ which we call the \emph{mass-decreasing flow}.\\
\end{itemize}
\end{definition}

Essentially, the first part of the definition means that we run a suitable Ricci flow with surgery for one unit of time (throwing away the pieces that bubble off). The second part of the definition means that we conformally rescale to a scalar flat metric. Finally, this process is iterated forever.\\

For definiteness, if $t$ is a surgery or rescaling time we denote by $(M(t),g(t))$ the \emph{pre}surgery, \emph{pre}rescaling manifold.\\

The following theorem shows that the mass-decreasing flow exists for all times and that it has the desired properties.

\begin{theorem}\label{thmLTE}
The mass-decreasing flow exists for all times, and preserves the asymptotic flatness and the nonnegative integrable scalar curvature. The mass is constant in the time intervals $t\in((k-1)\eps,k\eps)$ and jumps down by
\begin{equation}\label{deltam}
\delta m_k=-\int_{M}(8\abs{\D w_k}^2+Rw_k^2)dV
\end{equation} at the conformal rescaling times $t_k=k\eps$, where $w_k$ is the solution of
\begin{equation}\label{confeq}
\left(-8\Lap_{g(t_k)}+R_{g(t_k)}\right)w_k=0,\qquad\quad w_k\to 1\quad \textrm{at}\quad \infty.
\end{equation}
The monotonicity of the mass is strict as long as the metric is nonflat.
\end{theorem}

\begin{proof}
The surgical Ricci flow exists by Theorem \ref{BBMthm}. As we recalled above, there are only finitely many surgeries in finite time intervals and nonnegative scalar curvature is preserved.\\
The asymptotic flatness, the mass, and the integrable scalar curvature are all preserved along a nonsurgical Ricci flow with bounded curvature, see \cite{DM,OW}. Since $R\leq \Theta$ and since the surgeries only occur in regions with high curvature (i.e. in particular inside a compact region), these properties are also preserved along the surgical Ricci flow.\\
For the conformal rescaling part, writing $w_k=1+u_k$, we have to solve
\begin{equation}\label{equforu}
\left(-8\Lap_{g(t_k)}+R_{g(t_k)}\right)u_k=-R_{g(t_k)},\qquad\quad u_k\to 0\quad \textrm{at}\quad \infty.
\end{equation}
Since $R\geq 0$, the operator $\left(-8\Lap+R\right)$ is positive and thus invertible (viewed as operator between suitable weighted function spaces). In fact, we can solve (\ref{equforu}) with the estimate $u_k=O(r^{-1})$ at infinity. Consider the conformal metric $g_k=w_k^4g(t_k)$. Note that $(M(t_k),g_k)$ is an asymptotically flat manifold of order one with vanishing scalar curvature. Using the definition of the mass, the asymptotics $w_k-1=O(r^{-1})$ and $g_{ij}-\delta_{ij}=O(r^{-1})$, and writing $g_{ij}=g(t_k)_{ij}$ we compute
\begin{align}
m(g_k)&=\lim_{r\to\infty}\int_{S_r}w_k^{4}\left[\left(\partial_jg_{ij}-\partial_ig_{jj}\right)+\tfrac{4}{w_k}\left(g_{ij}\partial_jw_k-g_{jj}\partial_iw_k\right)\right]dA^i\\
&=\lim_{r\to\infty}\int_{S_r}\left[\left(\partial_jg_{ij}-\partial_ig_{jj}\right)+4\left(\delta_{ij}\partial_jw_k-\delta_{jj}\partial_iw_k\right)\right]dA^i\\
&=m(g(t_k))-8\lim_{r\to\infty}\int_{S_r}\partial_rw_kdA.
\end{align}
Furthermore, using partial integration and the asymptotics from above we compute
\begin{align}
&\int_M\left(8\abs{\nabla w_k}^2+Rw_k^2\right)dV=\lim_{r\to\infty}\int_{B_r}\left(8\abs{\nabla w_k}^2+Rw_k^2\right)dV\\
&\qquad=\lim_{r\to\infty}\int_{B_r}\left(8w_k(-\Lap w_k)+Rw_k^2\right)dV+\lim_{r\to\infty}\int_{S_r}8w_k\partial_rw_kdA\\
&\qquad=8\lim_{r\to\infty}\int_{S_r}\partial_rw_kdA,
\end{align}
where we also used equation (\ref{confeq}) in the last step. Putting everything together this implies
\begin{align}
\delta m_k=m(g_k)-m(g(t_k))=-\int_M\left(8\abs{\nabla w_k}^2+Rw_k^2\right)dV.
\end{align}
Finally, if $g_{k-1}$ is nonflat, then the scalar curvature becomes strictly positive under the Ricci flow (with surgery), and thus we conclude that $\delta m_k<0$.
\end{proof}

\begin{rem}
In fact, it is not really necessary to assume that the scalar curvature of the initial metric is integrable. If $\int_M RdV=\infty$ initially, then   the mass is infinite initially, but it becomes finite after one conformal rescaling.
\end{rem}

\section{Long-time behavior I}\label{longtime1}
Recall that the Ricci flow with surgery on a closed manifold that admits a metric with positive scalar curvature becomes extinct in finite time \cite{P3, CM1,CM2}. In a similar spirit, along the mass-decreasing flow wormholes pinch off and nontrivial spherical space forms bubble off in finite time.
\begin{theorem}\label{thmfinext}
There exists a $T<\infty$, such that $M(t)\cong\RR^3$ for $t>T$. In fact, one can take $T=\tfrac{A_0}{4\pi}$, where $A_0$ is the area of the largest outermost minimal two-sphere in $(M,g_0)$.
\end{theorem}
\begin{proof}
The idea is that minimal two-spheres shrink at rate at least $4\pi$. So for $t>T=\tfrac{A_0}{4\pi}$ the manifold $M(t)$ is diffeomorphic to $\RR^3$.\\
Let us now go through the details. Outermost minimal two-spheres exist by a result of Meeks-Simon-Yau \cite{MSY}. Instead of the area of the largest outermost two-sphere we actually consider the slightly different function
\begin{equation}
A(t):=\inf_\Omega\sup_{S_i}\abs{S_i},
\end{equation}
where $\Omega\subset M(t)$ is a manifold with boundary $\partial \Omega=S_1\cup\ldots\cup S_k$ that contains infinity and has the topology of a large disk with finitely many small disks removed. Along the mass-decreasing flow $A(t)$ satisfies the differential inequality,
\begin{equation}
\tfrac{d}{dt}A\leq -4\pi
\end{equation}
in the sense of the limsup of forward difference quotients, compare with \cite[Lemma 2.1]{CM1}. Indeed, this follows from a nice computation using the Gauss-Codazzi equation and the Gauss-Bonnet formula (note that the computation and the result simplify since $R\geq 0$). The surgeries and the conformal rescalings only help (since $w_k\leq 1$ by the maximum principle). Note that $A(0)\leq A_0$. The result follows.
\end{proof}

\begin{cor}\label{cortop}
The initial manifold had the diffeomorphism type
\begin{equation}\label{topequ}
M\cong \RR^3\#S^3/\Gamma_1\#\ldots\#S^3/\Gamma_k\#(S^1\times S^2)\#\ldots\#(S^1\times S^2).
\end{equation}
Conversely, any such manifold admits an asymptotically flat metric of order one with nonnegative integrable scalar curvature (in fact there exists an asymptotically flat metric on $M$ with vanishing scalar curvature).
\end{cor}

\begin{proof}
The proof is along the lines of Perelman \cite{P2} and Bessi\`eres-Besson-Maillot \cite{BBM}. When flowing from $M(0)$ to $M(t)$, the topology can only change for the following two reasons. First, it can happen that compact components with positive scalar curvature are removed. These components are diffeomorphic to a connected sum of spherical space forms and $S^1\times S^2$ pieces \cite{P2}. Second, the surgery can change the topology by pinching off wormholes. This can be seen by moving from the center of the $\delta$-neck to the left and to the right until arriving in regions with lower curvature, the swept out manifold being diffeomorphic to $S^2\times \RR$. Therefore, the initial manifold has the topology as stated in (\ref{topequ}).\\
Conversely, as proved by Schoen-Yau \cite{SY2} and Gromov-Lawson \cite{GL}, one can construct a metric of positive scalar curvature on 
\begin{equation}
S^3/\Gamma_1\#\ldots\#S^3/\Gamma_k\#(S^1\times S^2)\#\ldots\#(S^1\times S^2).
\end{equation}
One can then obtain an asymptotically flat metric of order one with vanishing scalar curvature on
\begin{equation}
\RR^3\#S^3/\Gamma_1\#\ldots\#S^3/\Gamma_k\#(S^1\times S^2)\#\ldots\#(S^1\times S^2),
\end{equation}
by `stereographic projection' using the Greens-function of the conformal Laplacian, $-8\Lap+R$, see e.g. Lee-Parker \cite{LP}.
\end{proof}

\begin{rem}
There are two other ways how the assertion of Corollary \ref{cortop} can be proved. The first one is to compactify the initial manifold to a closed manifold with positive scalar curvature and to use Perelman's result for closed manifolds with positive scalar curvature \cite{P1,P2}. The second one is to combine the result of Schoen-Yau \cite{SY2} with a lot of three-manifold topology. One needs the spherical space form conjecture proved by Perelman \cite{P1,P2}, and also the solution of the surface subgroup conjecture. The proof of the surface subgroup conjecture is obtained by combining Perelman's solution of the geometrization conjecture \cite{P1,P2}, and the recent work of Kahn-Markovic \cite{KM}.
\end{rem}

\section{Long-time behavior II}\label{longtime2}

To investigate the geometric-analytic aspects of the long-time behavior we will follow the general principle that monotonicity formulas are a very useful tool. However, when trying to follow this principle one encounters the fundamental problem that Perelman's $\lambda$-energy \cite{P1} adapted as it stands,
\begin{equation}
\lambda(g):=\inf_{w:\int w^2=1} \int_M\left(4\abs{\D w}^2+Rw^2\right)dV,
\end{equation}
is monotone only in a very trivial way. Namely, minimizing sequences $w_i$ escape to infinity and the value of $\lambda$ is identically zero along the flow. We overcome this difficulty by considering instead the following variant of Perelman's $\lambda$-functional,
\begin{equation}
\lambda_{AF}(g):=\inf_{w:w\to 1} \int_M\left(4\abs{\D w}^2+Rw^2\right)dV,
\end{equation}
where the infimum is now taken over all $w\in C^\infty(M)$ such that $w=1+O(r^{-1})$ at infinity. Unless the scalar curvature vanishes identically, the value of $\lambda_{AF}$ is strictly positive in our setting of asymptotically flat manifolds with nonnegative scalar curvature.\\

We have introduced the energy-functional $\lambda_{AF}$ in our recent work \cite{H}, where we also observed that $\lambda_{AF}$ gives a lower bound for the mass, i.e. we have the inequality 
\begin{equation}\label{lowerbdd}
m(g)\geq \lambda_{AF}(g).
\end{equation}
Furthermore, the (renormalized) Perelman-energy also plays an important role in questions concerning the stability of Ricci-flat spaces, see e.g. \cite{CHI,Ses,SSS,Hstab,H,HHS} for more information on this aspect. Finally, the Perelman-energy and its variant $\bar{\lambda}=\sup\lambda V^{2/3}$ numerically characterize the long-time behavior of the Ricci flow with surgery on closed three-manifolds \cite{P2}.\\

Having completed this short overview and motivation, we will now prove that $\lambda_{AF}$ satisfies a Perelman-type monotonicity formula.

\begin{theorem}\label{thmmon}
Away from the conformal rescaling and surgery times, we have the monotonicity formula
\begin{equation}\label{monform}
\tfrac{d}{dt}\lambda_{AF}(g(t))=2\int_M\abs{\Ric+\D^2f}^2e^{-f}dV\geq 0,
\end{equation}
where $f$ is the unique solution of
\begin{equation}\label{mineq}
\left(-4\Lap+R\right)e^{-f/2}=0,\qquad\quad f\to 0\quad \textrm{at}\quad \infty.
\end{equation}
\end{theorem}

\begin{rem}
To avoid problems with modifying the Ricci flow by a family of diffeomorphisms in the noncompact setting, we will prove our monotonicity formula by a direct computation. This computation is quite different than Perelman's original computation in the compact setting (note however, that the Bianchi identity used below is of course just another manifestation of diffeomorphism invariance).
\end{rem}

\begin{proof}
Substituting $w=e^{-f/2}$, the definition of $\lambda_{AF}$ can be rewritten as
\begin{equation}
\lambda_{AF}(g)=\inf_{f:f\to 0} \int_M\left(R_g+\abs{\D f}_g^2\right)e^{-f}dV_g,
\end{equation}
where the infimum is taken over all $f\in C^\infty(M)$ such that $f=O(r^{-1})$ at infinity. By a similar argument as in the proof of Theorem \ref{thmLTE}, there exists a unique minimizer. The time derivative of $\lambda_{AF}$ along the Ricci flow equals
\begin{equation}
\tfrac{d}{dt}\lambda_{AF}(g(t))=\int_M\left[\Lap R+2\abs{\Ric}^2+2\Ric(\D f,\D f)-(R+\abs{\D f}^2)R\right]e^{-f}dV,
\end{equation}
where $f$ is the minimizer (at the time in consideration). Here, the first two terms come from the evolution of the scalar curvature, the third term comes from the evolution of the inverse metric, and the last term comes from the evolution of the volume element. The monotonicity formula will now follow from a computation using partial integrations, the Bianchi identity, and the equation (\ref{mineq}) for the minimizer, which can be rewritten as
\begin{equation}\label{minimeq}
2\Lap f-\abs{\D f}^2+R=0.
\end{equation}
The partial integrations can be justified using the decay estimates $g_{ij}-\delta_{ij}=O(r^{-1})$, $f=O(r^{-1})$, and the computation consists of the following pieces. First, we have the Bochner-type identity
\begin{multline}
\int_M\abs{\D^2 f}^2e^{-f}dV=\\
\int_M\left[-\langle \D f,\D\Lap f\rangle-\Ric(\D f,\D f)+\D^2f(\D f,\D f)\right]e^{-f}dV.
\end{multline}
Second, using the Bianchi identity we obtain
\begin{equation}
\int_M\langle\Ric,\D^2f\rangle e^{-f}dV=\int_M\left[\Ric(\D f,\D f)-\tfrac{1}{2}\langle \D R,\D f\rangle\right]e^{-f}dV.
\end{equation}
Third, using (\ref{minimeq}) we get the pointwise identity
\begin{equation}
\D^2f(\D f,\D f)-\langle\D f,\D\Lap f\rangle = \tfrac{1}{2}\langle\D f,\D R\rangle.
\end{equation}
Fourth, we have the partial integration formula
\begin{equation}
\int_M\Lap Re^{-f}dV=\int_M\langle\D R,\D f\rangle e^{-f}dV=\int_M (\abs{\D f}^2-\Lap f)Re^{-f}dV.
\end{equation}
Fifth and finally, equation (\ref{minimeq}) can be rewritten as
\begin{equation}
2(\abs{\D f}^2-\Lap f)=R+\abs{\D f}^2.
\end{equation}
Putting everything together (starting from the right hand side of (\ref{monform}) for convenience), the monotonicity formula follows.
\end{proof}

\begin{rem}\label{surgthm}
Choosing the surgery parameter $\delta$ small enough, $\lambda_{AF}$ can be made almost monotone at the surgery times (compare with \cite{P2,KL}).
\end{rem}

\begin{rem}
Immediately after the conformal rescaling we have $f=0$ for the minimizer, and thus
\begin{equation}
\tfrac{d}{dt}|_{t_k+}\lambda_{AF}(g(t))=2\int_M\abs{\Ric}^2dV.
\end{equation}
 Moreover, looking at the expression for a Schwarzschild-end suggests that
\begin{equation}
\int_M\abs{\Ric}^2dV\sim m^2,
\end{equation}
and thus $\lambda_{AF}\sim \eps m^2$ after time $\eps$, i.e. we expect that $\lambda_{AF}$ is proportional to the square of the mass at the time $t_{k+1}=(k+1)\eps$.
\end{rem}

Under an a priori assumption, an inequality between $\lambda_{AF}$ and the mass that is motivated by the above remark and complements the inequality (\ref{lowerbdd}), we can prove that the flow indeed squeezes out all the initial mass.

\begin{theorem}\label{thmapriori}
Let $(M(t),g(t))_{t\in[0,\infty)}$ be a solution of the mass-decreasing flow and assume a priori there exist a constant $c>0$, such that
$\lambda_{AF}(g(t_k))\geq c m(g(t_k))^2$ for all positive integers $k$. Then there exists a constant $C<\infty$ such that $m(g(t))\leq C/t$. In particular, the mass-decreasing flow squeezes out all the initial mass, i.e. $\lim_{t\to\infty}m(g(t))=0$.
\end{theorem}

\begin{proof}
Using Theorem \ref{thmLTE}, the definition of $\lambda_{AF}$, the inequality $4\leq 8$, and the a priori assumption, a computation gives
\begin{align}
m(g(t_{k+1}))-m(g(t_k))&=-\int_{M}(8\abs{\D w_k}^2+Rw_k^2)dV\nonumber\\
&\leq - \lambda_{AF}(g(t_k))\leq - c m(g(t_k))^2.
\end{align}
This implies that there exists a constant $C<\infty$ such that $m(g(t))\leq C/t$. Using this and the positive mass theorem, we conclude that the flow indeed squeezes out all the initial mass.
\end{proof}

\section{The continuum limit}\label{contlimit}
We now derive the limiting equations that formally arise when the iteration parameter $\eps$ is sent to zero. In the following the symbol $\cong$ denotes equality modulo terms of order $\eps^2$ and higher (assuming curvature bounds a priori, the error terms could be estimated explicitly).\\

From the evolution equation $\partial_tR=\Lap R+2\abs{\Ric}^2$ and $R_{g_{k-1}}=0$ we get
\begin{equation}
R_{g(t_k)}\cong 2\eps\abs{\Ric_{g(t_{k-1})}}^2.
\end{equation}
Then, solving $(-8\Lap_{g(t_k)}+R_{g(t_k)})w_k=0$ with $w_k\to 1$ at infinity gives
\begin{equation}
w_k\cong 1+\frac{\eps}{4}\Lap^{-1}_{g(t_k)}\abs{\Ric_{g(t_{k-1})}}^2.
\end{equation}
To first order the metric $g_k=w_k^4g(t_k)$ equals
\begin{equation}
g_k\cong\left(1+\eps\Lap^{-1}_{g(t_k)}\abs{\Ric_{g(t_{k-1})}}^2\right)g(t_k),
\end{equation}
and using also $g(t_k)\cong g_{k-1}-2\eps\Ric_{g_{k-1}}$ this becomes
\begin{equation}
g_k\cong g_{k-1}-2\eps\Ric_{g_{k-1}}+\eps\Lap_{g_{k-1}}^{-1}\abs{\Ric_{g_{k-1}}}^2g_{k-1},
\end{equation}
where we also approximated the inverse Laplacian and the Ricci curvature dropping terms of higher order. Thus, the limiting evolution equation is
\begin{equation}
\partial_t g=-2\Ric+\Lap^{-1}\abs{\Ric}^2 g,
\end{equation}
which is the Ricci flow modified by a nonlocal conformal factor which has the effect of projecting to the space of scalar flat metrics. Short time existence for this nonlocal flow was proved in \cite[Thm 1.3]{LQZ}.\\

Furthermore, note that the quantity $m-\lambda_{AF}$ is monotone at all times (i.e. at the conformal rescaling times and also in between). In the formal limit $\eps\to 0$ the quantity $\lambda_{AF}$ vanishes identically and our monotonicity formulas boil down to the formula
\begin{equation}\label{monformu}
\partial_t m=-2\int_M\abs{\Ric}^2 dV.
\end{equation}
This monotonicity formula also appear in \cite[Thm 1.4]{LQZ}.

\section{Problems and Questions}\label{openprob}
We conclude this article with a list of open problems and questions. Some of them came up in discussions with Lars Andersson, Hugh Bray, Gerhard Huisken, and Tom Ilmanen.
\begin{itemize}
\item Can one get rid of the a priori assumption relating the mass and the Perelman-energy?
\item Can the mass-decreasing flow be used to give an independent proof of the positive mass theorem?
\item Is there some clever argument in higher dimensions?
\item Is there some relationship with the Penrose inequality?
\end{itemize}
Let us comment on the first to questions. Assuming the positive mass theorem instead of proving it for the moment, motivated by (\ref{monformu}) we expect a space-time integral bound
\begin{equation}\label{intbound}
\int_0^\infty\int_M\abs{\Ric}^2dVdt\leq C
\end{equation}
for the mass-decreasing flow, at least when $\eps$ is chosen small enough. Thus the limit for $t\to\infty$ is flat in some integral sense. To answer the first question, one has to improve this into a convergence sufficiently strong to conclude that the mass limits to zero. To answer the second question, one could try to couple this argument with the derivation of (\ref{intbound}).

\end{document}